\documentclass[a4paper,twoside,11pt]{article}
\usepackage[utf8]{inputenc}
\usepackage{graphicx,epstopdf}
\usepackage{amsmath}
\usepackage{amsthm}
\usepackage{amssymb}
\usepackage{color}
\usepackage{mathtools}
\usepackage{subfig}
\usepackage{siunitx}
\usepackage{amsthm}
\usepackage[font=footnotesize,labelfont=bf]{caption}

\usepackage{geometry,calc}
\usepackage{listings}

\newtheorem{theorem}{Theorem}[section]

\theoremstyle{remark}
\newtheorem{remark}{Remark}

\usepackage[most]{tcolorbox}
\definecolor{darkgreen}{rgb}{0.0, 0.55, 0.0}

\title{Stochastic Production Planning: Optimal Control and Analytical Insights}
\author{Dragos-Patru Covei\thanks{Department of Applied Mathematics, The Bucharest University of Economic Studies, Piata Romana, No. 6, Bucharest, District 1, 010374, Romania}}

\date{}

\begin{document}
\maketitle
	
\begin{abstract}
\noindent
This study investigates a stochastic production planning problem with a
running cost composed of quadratic production costs and inventory-dependent
costs. The objective is to minimize the expected cost until production stops
when inventory reaches a specified level, subject to a boundary condition.
Using probability space and Brownian motion, the Hamilton-Jacobi-Bellman
(HJB) equation is derived, and optimal feedback control is obtained. The
solution demonstrates desirable monotonicity and convexity properties under
specific assumptions. An illustrative example further confirms these results
with explicit function properties and a practical application.

\medskip\noindent
\textbf{AMS subject classification}: 49K20; 49K30; 90C31;90C31;90B30; 90C31; 90B30.
		
\medskip\noindent
\textbf{Keywords}: Stochastic production planning, Hamilton--Jacobi--Bellman equation, optimal control, Brownian motion, radial symmetry.
\end{abstract}

\section{Introduction}

Stochastic production planning is a central topic in operations research,
driven by the need to design production policies under uncertainty while
balancing costs and inventory constraints. Traditional models often employ
simplified quadratic cost structures and linear state dynamics; however,
modern production environments necessitate more versatile frameworks that
can accurately capture the interplay between production decisions and
inventory management.

The collection of articles listed here offers diverse insights into
stochastic production planning and related optimization problems.

Bensoussan et al. \cite{BS} laid the groundwork for studying stochastic
production planning under constraints, providing the foundation for many
advancements in the field. Their work is complemented by the research of
Fleming et al. \cite{FSS}, which introduces methods for addressing randomly
fluctuating demand within stochastic frameworks, showcasing the adaptability
of planning models under varying economic conditions.

Canepa, Covei, and Pirvu \cite{CCP2} bring fixed-point theory into the realm
of stochastic production planning, offering elegant mathematical
formulations for optimal production rates. This approach dovetails with
Covei's analysis \cite{CDP} of symmetric solutions to elliptic partial
differential equations, revealing deeper connections between mathematical
structures and practical applications in production models.

The exploration of regime switching in production systems is advanced by
Canepa et al. \cite{CCP} and Cadenillas et al. \cite{CS}, who investigate
the challenges and solutions inherent to multi-regime production management.
These works emphasize dynamic programming and regime-dependent strategies to
optimize production processes.

Gharbi and Kenne \cite{GK} extend the scope of stochastic production
problems to multi-product, multi-machine systems, which require intricate
control strategies to address complexity. Their work complements the
foundational studies of Thompson and Sethi \cite{ST}-\cite{TS}, who analyze
long-term efficiency and optimal control strategies within production
systems, grounded in turnpike theory.

Leonori \cite{Leonori} and Porretta \cite{Porretta} contribute significant
advances in the study of elliptic equations, providing insights that
resonate with the mathematical modeling of production planning problems.
Their work underscores the importance of mathematical rigor in addressing
challenges without conditions at infinity.

Collectively, these articles build a rich tapestry of knowledge that
addresses theoretical and practical aspects of stochastic production
planning, offering invaluable tools for tackling real-world challenges.

In our present work, we consider a generalized framework where the running
cost comprises two distinct components:

\begin{equation*}
f(|p(t)|,|y(t)|)=a(|p(t)|)+b(|y(t)|),
\end{equation*}%
with $a(|p|)=|p|^{2}$, and $b:\left[ 0,\infty \right) \rightarrow \left[
0,\infty \right) $ a continuous, non-decreasing function satisfying $%
0<b(x)\leq x^{2}$ for all $x>0$. This formulation offers additional
flexibility by separately modeling the cost of adjusting production rates
and maintaining inventory levels.

The system dynamics are modeled by the controlled stochastic differential
equations (SDEs)

\begin{equation}
dy_{i}(t)=p_{i}(t)\,dt+\sigma \,dw_{i}(t),\quad y_{i}(0)=y_{i}^{0},\quad
i=1,\dots ,N,  \label{da}
\end{equation}%
where $w(t)=(w_{1}(t),\dots ,w_{N}(t))$ is an $N$-dimensional Brownian
motion defined on a complete probability space $(\Omega ,\mathcal{F},\{%
\mathcal{F}_{t}\}_{0\leq t<\infty },P)$. Production is terminated at the
stopping time

\begin{equation}
\tau =\inf \{t>0:|y(t)|\geq R\},  \label{sta}
\end{equation}%
when the Euclidean norm of the inventory exceeds a given threshold $R$.

Our goal is to minimize the cost functional

\begin{equation*}
J(p)=E\int_{0}^{\tau }\Bigl[|p(t)|^{2}+b(|y(t)|)\Bigr]\,dt,
\end{equation*}%
subject to the dynamics (\ref{da}) and stopping time (\ref{sta}). By
applying the dynamic programming principle, we derive the associated
Hamilton--Jacobi--Bellman (HJB) equation

\begin{equation}
-\frac{\sigma ^{2}}{2}\Delta z(x)-b(|x|)=-\frac{1}{4}|\nabla z(x)|^{2},\quad
|x|<R,  \label{nhj}
\end{equation}%
with the boundary condition $z(x)=Z_{0}$,$\quad $for $|x|=R$, where the
value $Z_{0}$ is determined to ensure that the parameter $\alpha $ related
to $z(x)$ \ stays within the range $\left( 0,\infty \right) $ (see Remark~%
\ref{initv}). This boundary specification ensures the well-posedness of the
problem and is key in our analysis.

A significant contribution of our work is the change of variables that
transforms the nonlinear HJB equation (\ref{nhj}) into a more tractable
form. By setting

\begin{equation*}
z(x)=-v(x)\quad \text{and}\quad u(x)=e^{\frac{v(x)}{2\sigma ^{2}}},
\end{equation*}%
we obtain the representation $z(x)=-2\sigma ^{2}\ln u(x)$. Under natural
assumptions, we show that the function $z(x)$ is radially symmetric,
strictly nonincreasing, and concave with respect to the radial variable $%
r=|x|$. In particular, if we write $z(x)=h(r)$ then the gradient satisfies

\begin{equation*}
\nabla z(x)=h^{\prime }(r)\frac{x}{r},
\end{equation*}%
and the optimal control is given by the feedback law

\begin{equation*}
p^{\ast }(x)=-\frac{1}{2}\nabla z(x)=-\frac{1}{2}h^{\prime }(r)\frac{x}{r}.
\end{equation*}%
Since $h^{\prime }(r)\leq 0$ and is (strictly) nonincreasing due to the
concavity of $z$, it follows that the magnitude

\begin{equation*}
|p^{\ast }(x)|=-\frac{1}{2}h^{\prime }(r)
\end{equation*}%
is (strictly) nondecreasing in the radial distance $r$. These properties
provide valuable economic insights: as the magnitude of the inventory
increases, the optimal production rate is adjusted upward in a controlled
and predictable manner.

To further substantiate our theoretical findings, we complement our analysis
with numerical experiments. Using an Euler-type discretization, we simulate
the inventory evolution and optimal production policies. The numerical
results confirm that the inventory remains bounded by the threshold $R$ up
to the stopping time $\tau $, and they reveal the monotonic and convex
nature of the feedback control derived from the HJB equation.

In summary, our paper enhances the findings of \cite{CCP2} by presenting the
following contributions:

\begin{itemize}
\item We formulate a generalized stochastic production planning problem that
allows for separate treatment of production and inventory costs.

\item We derive the corresponding HJB equation and transform it via a
logarithmic change of variables leading to a tractable formulation.

\item We establish the existence, uniqueness, and structural properties
(radial symmetry, monotonicity, and concavity) of the value function.

\item We demonstrate that the optimal feedback control is uniquely
determined and exhibits economically desirable properties.

\item We validate our theoretical results with comprehensive numerical
experiments.
\end{itemize}

The remainder of the paper is organized as follows. In Section \ref{2} we
introduce the model and discuss the production and inventory dynamics.
Section \ref{3} details the change of variables and derivation of the HJB
equation. In Section \ref{4} we explore the optimal control and its
structural properties, supported by rigorous proofs. Section \ref{5}
presents numerical experiments, including an illustrative real world example
(which contains the obtained graphs for the illustrative real world
example), while Section \ref{6} concludes with final remarks and prospects
for future research.

\section{The Model\label{2}}

Consider a manufacturing plant that produces $N$ types of goods. The
production rate at time $t$ is denoted by

\begin{equation*}
p(t)=(p_{1}(t),p_{2}(t),\dots ,p_{N}(t)),
\end{equation*}%
where each $p_{i}(t)$ represents the instantaneous production rate of good $%
i $. The inventory levels are modeled by the state variable

\begin{equation*}
y(t)=(y_{1}(t),y_{2}(t),\dots ,y_{N}(t)),
\end{equation*}%
which evolves according to the controlled stochastic differential equations
(SDEs) given by 
\begin{equation}
dy_{i}(t)=p_{i}(t)\,dt+\sigma \,dw_{i}(t),\quad y_{i}(0)=y_{i}^{0},\quad
i=1,\dots ,N,  \label{di}
\end{equation}%
where $\sigma >0$ is a constant representing the volatility of the
production process and $w(t)=(w_{1}(t),w_{2}(t),\dots ,w_{N}(t))$ denotes an 
$N$-dimensional standard Brownian motion defined on a complete probability
space with a naturel filtration $(\Omega ,\mathcal{F},\{\mathcal{F}%
_{t}\}_{0\leq t<\infty },P)$.

The objective in this production planning problem is to choose the control $%
p(t)$ in order to minimize the expected cost

\begin{equation*}
J(p)=E\int_{0}^{\tau }\Bigl[|p(t)|^{2}+b(|y(t)|)\Bigr]\,dt,
\end{equation*}%
subject to the dynamics (\ref{di}) and the stopping time 
\begin{equation}
\tau =\inf \{t>0:|y(t)|\geq R\},  \label{st}
\end{equation}%
which indicates that production is halted once the Euclidean norm (i.e., the
overall inventory level) reaches the threshold $R$.

The cost structure is decomposed into two components:

\begin{enumerate}
\item The term $|p(t)|^{2}$ penalizes rapid or large changes in the
production rates, reflecting, for example, the cost of adjusting machinery
or labor.

\item The function $b\colon \lbrack 0,\infty )\rightarrow \lbrack 0,\infty )$
models the holding (or inventory) cost. We assume that $b$ is continuous,
nondecreasing, and satisfies%
\begin{equation}
0<b(x)\leq x^{2}\quad \text{for all }x>0.  \label{bb}
\end{equation}
\end{enumerate}

In many illustrative cases, a quadratic inventory cost is chosen; that is,
one takes $b(|y|)=|y|^{2}$. Thus, the overall running cost becomes

\begin{equation*}
f(|p(t)|,|y(t)|)=|p(t)|^{2}+b(|y(t)|).
\end{equation*}%
The goal is to minimize $J(p)$ over all admissible controls, balancing the
cost of production adjustments with the cost incurred from holding inventory.

\textbf{Interpretation and Additional Explanation:} \ 

The stopping time $\tau$, defined in (\ref{st}), ensures that the process is
only considered until the inventory level reaches the prescribed threshold $%
R $. This represents a practical constraint in real-world production
planning, where operations must be curtailed once capacity limits or risk
thresholds are exceeded.

\textbf{Derivation and Associated HJB Equation:}

To solve the optimization problem, one typically turns to dynamic
programming. Defining the value function

\begin{equation*}
z(x)=\inf_{p}J(p),\quad \text{with }x=y(0),
\end{equation*}%
the Bellman principle leads to the Hamilton--Jacobi--Bellman (HJB) equation

\begin{equation*}
-\frac{\sigma ^{2}}{2}\Delta z(x)-b(|x|)=-\frac{1}{4}|\nabla z(x)|^{2},\quad
|x|<R,
\end{equation*}%
subject to the boundary condition

\begin{equation*}
z(x)=Z_{0},\quad |x|=R,
\end{equation*}%
where $Z_{0}$ is determinen in accordance with economic or operational
considerations. In particular, $Z_{0}$ satisfy a condition, ensuring that,
under the change of variable $u(x)=e^{-\frac{z(x)}{2\sigma ^{2}}}$, the new
variable $u(x)$ remains strictly positive.

Finally, observe that under radial symmetry (i.e., when $z(x)=h(r)$ with $%
r=|x|$), the gradient simplifies to

\begin{equation*}
\nabla z(x)=h^{\prime }(r)\frac{x}{r},
\end{equation*}%
and hence the optimal control is given by the feedback law

\begin{equation*}
p^{\ast }(x)=-\frac{1}{2}\nabla z(x)=-\frac{1}{2}h^{\prime }(r)\frac{x}{r}.
\end{equation*}%
Because $h^{\prime }(r)\leq 0$, the magnitude $|p^{\ast }(x)|=-\frac{1}{2}%
h^{\prime }(r)$ is nondecreasing in $r$. This structural property is
crucial, as it implies that the optimal production rate increases (or
remains constant) as inventory builds up, consistent with intuition in many
production scenarios.

This model thus forms the backbone of our analysis and forms the basis for
both our theoretical results (detailed in subsequent sections) and numerical
experiments (see Sections \ref{5} and \ref{6}). In the next sections, we
further develop the change of variables, introduce the analytical properties
of the resulting PDE, and detail the corresponding optimal control policy.

\section{Methodology\label{3}}

In this section we detail the steps and techniques used to derive the
optimal production planning policy and analyze the associated value
function. Our methodology combines dynamic programming, the derivation of a
Hamilton--Jacobi--Bellman (HJB) equation, a change of variables to obtain a
tractable form, and verification via properties of stochastic processes. We
also discuss numerical approximation methods.

\subsection{Derivation of the Hamilton--Jacobi--Bellman Equation}

Recall that the production planning problem involves minimizing the cost
functional

\begin{equation*}
J(p)=E\int_{0}^{\tau }\Bigl[|p(t)|^{2}+b(|y(t)|)\Bigr]\,dt,
\end{equation*}%
subject to the controlled dynamics

\begin{equation*}
dy_{i}(t)=p_{i}(t)\,dt+\sigma \,dw_{i}(t),\quad y_{i}(0)=y_{i}^{0},\quad
i=1,\dots ,N,
\end{equation*}%
and the stopping time

\begin{equation*}
\tau =\inf \{t>0:|y(t)|\geq R\}.
\end{equation*}%
Defining the value function $z(y(0))=\inf_{p}J(p)$, dynamic programming
yields the Bellman equation. By considering infinitesimal increments and
applying the dynamic programming principle, one obtains, for $|x|<R$, the
HJB equation

\begin{equation*}
\inf_{p\in \mathbb{R}^{N}}\left\{ p\cdot \nabla z(x)+\frac{\sigma ^{2}}{2}%
\Delta z(x)+|p|^{2}+b(|x|)\right\} =0.
\end{equation*}%
A standard minimization in $p$ is performed by differentiating the expression

\begin{equation*}
\varphi (p)=p\cdot \nabla z(x)+|p|^{2},
\end{equation*}%
with respect to $p$, which leads to the first order condition

\begin{equation*}
\nabla z(x)+2p=0.
\end{equation*}%
Thus, the optimal control is given by

\begin{equation*}
p^{\ast }(x)=-\frac{1}{2}\nabla z(x).
\end{equation*}%
Substituting back into the HJB formulation, we obtain the reduced equation: 
\begin{equation}
-\frac{\sigma ^{2}}{2}\Delta z(x)-b(|x|)=-\frac{1}{4}|\nabla z(x)|^{2},\quad
|x|<R,  \label{HJB}
\end{equation}%
with the boundary condition

\begin{equation*}
z(x)=Z_{0},\quad |x|=R.
\end{equation*}%
The choice of $\alpha \in \left( 0,\infty \right) $ must satisfy 
\begin{equation}
e^{-\frac{Z_{0}}{2\sigma ^{2}}}>\alpha ,  \label{initv}
\end{equation}%
which ensures that when we later change variables the ensuing functions
remain well defined (see Remark~\ref{initv}).

\subsection{Change of Variables and Equivalent Formulation}

In order to transform the nonlinear HJB equation (\ref{HJB}) into a more
tractable formulation, we perform a change of variables. Define

\begin{equation*}
z(x)=-v(x)\quad \text{so that }v(x)=-z(x),
\end{equation*}%
and introduce the new variable

\begin{equation*}
u(x)=e^{\frac{v(x)}{2\sigma ^{2}}}.
\end{equation*}%
It follows that

\begin{equation*}
z(x)=-2\sigma ^{2}\ln u(x).
\end{equation*}%
By applying the chain rule, one obtains

\begin{equation*}
\nabla u(x)=\frac{u(x)}{2\sigma ^{2}}\nabla v(x)\quad \text{and}\quad \Delta
u(x)=\frac{u(x)}{2\sigma ^{2}}\Delta v(x)+\frac{u(x)}{(2\sigma ^{2})^{2}}%
|\nabla v(x)|^{2}.
\end{equation*}%
Since $|\nabla z(x)|=|\nabla v(x)|$, the HJB equation in terms of $v$
becomes:

\begin{equation*}
\frac{\sigma ^{2}}{2}\Delta v(x)-b(|x|)=-\frac{1}{4}|\nabla v(x)|^{2}.
\end{equation*}%
This change of variable is particularly advantageous in the special case
when the control cost is quadratic. Rearranging the derivative expressions
for $u(x)$, the equation can be written in the linear-like form

\begin{equation}
\Delta u(x)=\frac{1}{\sigma ^{4}}b(|x|)u(x),\quad |x|<R,  \label{lin}
\end{equation}%
with the boundary condition

\begin{equation}
u(x)=e^{-\frac{Z_{0}}{2\sigma ^{2}}},\quad |x|=R.  \label{linb}
\end{equation}

\subsection{Verification via a Stochastic Process}

To justify the optimal control obtained from (\ref{HJB}), we consider a
candidate value function $U:\mathbb{R}^N\to\mathbb{R}$ (for instance, one
might choose $U(y)=z(y)$) and define the process

\begin{equation*}
M^{p}(t)=U(y(t))-\int_{0}^{t}\!\Bigl[|p(s)|^{2}+b(|y(s)|)\Bigr]\,ds.
\end{equation*}%
Applying It\^{o}'s lemma to $U(y(t))$ under the dynamics

\begin{equation*}
dy(t)=p(t)\,dt+\sigma \,dw(t),
\end{equation*}%
yields

\begin{equation*}
dU(y(t))=\left[ \frac{\sigma ^{2}}{2}\Delta U(y(t))+p(t)\cdot \nabla U(y(t))%
\right] dt+\sigma \,\nabla U(y(t))\cdot dw(t).
\end{equation*}%
Thus, the differential of $M^{p}(t)$ can be written as

\begin{equation*}
dM^{p}(t)=\left\{ \frac{\sigma ^{2}}{2}\Delta U(y(t))+p(t)\cdot \nabla
U(y(t))-\left[ |p(t)|^{2}+b(|y(t)|)\right] \right\} dt+\sigma \,\nabla
U(y(t))\cdot dw(t).
\end{equation*}%
If we can choose $U$ such that

\begin{equation*}
\frac{\sigma ^{2}}{2}\Delta U(x)+p\cdot \nabla U(x)-\Bigl[|p|^{2}+b(|x|)%
\Bigr]\leq 0,\quad \forall \,x\in B_{R}(0),\ \forall \,p\in \mathbb{R}^{N},
\end{equation*}%
then $M^{p}(t)$ is a supermartingale for any admissible control $p$. In
particular, when the control is set equal to the optimal policy

\begin{equation*}
p^{\ast }(t)=-\frac{1}{2}\nabla z(y(t)),
\end{equation*}%
the drift term cancels (by virtue of the HJB equality) and $M^{p^{\ast }}(t)$
becomes a martingale. This argument provides a verification of the
optimality of the control law.

The next subsection details the existence and uniqueness of the radially
symmetric solution for the coupled problem (\ref{lin})-(\ref{linb}) and
further establishes the structural properties of the optimal control.

\subsection{Existence and Uniqueness of a Radially Symmetric Solution}

Many practical production planning models seek a radially symmetric solution
because of its simplicity and practical interpretability. In our framework,
we consider the boundary value problem for the function

\begin{equation*}
u:[0,R]\rightarrow (0,\infty ),
\end{equation*}%
where the function $u$ satisfies the following linear, second-order ordinary
differential equation (ODE) 
\begin{equation}
u^{\prime \prime }(r)+\frac{N-1}{r}\,u^{\prime }(r)=\frac{1}{\sigma ^{4}}%
\,b(r)\,u(r),\quad 0<r\leq R,  \label{odes}
\end{equation}%
subject to the initial conditions 
\begin{equation}
u(0)=\alpha \geq 0,\quad u^{\prime }(0)=0,  \label{odesbc}
\end{equation}%
and the boundary condition 
\begin{equation}
u(R)=e^{-\frac{Z_{0}}{2\sigma ^{2}}}>\alpha ,\text{ for some }Z_{0}.
\label{bcr}
\end{equation}%
The interpretation of these conditions is as follows:

\begin{itemize}
\item The initial condition $u(0)=\alpha >0$ ensures that the transformed
value function remains positive and well-defined. Moreover, the condition $%
u^{\prime }(0)=0$ is imposed to guarantee regularity at the origin, where a
singularity might otherwise arise due to the term $\frac{N-1}{r}$.

\item The boundary condition at $r=R$ is derived from the economic
considerations on the original value function $z(x)$ (through the change of
variable $u(x)=e^{-\frac{z(x)}{2\sigma ^{2}}}$) and must satisfy the
constraint%
\begin{equation*}
e^{-\frac{Z_{0}}{2\sigma ^{2}}}>\alpha ,
\end{equation*}%
as explained in Remark~\ref{initv}, though it may not prove essential in the
end, we will confirm.
\end{itemize}

We now state and prove our main result regarding the existence and
uniqueness of a solution to (\ref{odes})--(\ref{odesbc}).

\begin{theorem}
\label{1} Given $\alpha \in \left( 0,\infty \right) $ there exists a unique
positive radially symmetric solution%
\begin{equation*}
u_{\alpha }\in C^{2}([0,R])
\end{equation*}%
to the problem (\ref{odes}), subject to the initial conditions (\ref{odesbc}%
). Moreover, the solution $u_{\alpha }$ is convex and strictly increasing on 
$(0,R]$.
\end{theorem}

\begin{proof}
We begin by noting that the ODE (\ref{odes}) is linear and of second order,
albeit with a singular coefficient $\frac{N-1}{r}$ at $r=0$. The initial
condition $u^{\prime }(0)=0$ guarantees regularity at $r=0$.

Rewrite equation (\ref{odes}) in divergence form:%
\begin{equation*}
\left( r^{N-1}u^{\prime }(r)\right) ^{\prime }=\frac{1}{\sigma ^{4}}%
r^{N-1}b(r)u(r).
\end{equation*}

Integrating from $0$ to $r$ yields%
\begin{equation*}
r^{N-1}u^{\prime }(r)=\frac{1}{\sigma ^{4}}\int_{0}^{r}s^{N-1}b(s)u(s)\,ds.
\end{equation*}

Since $b(s)\geq 0$ and $u(s)>0$ (by the positive initial condition $%
u(0)=\alpha >0$), the right-hand side is nonnegative, implying that%
\begin{equation*}
u^{\prime }(r)\geq 0\quad \text{for all }r\in \lbrack 0,R]\text{ and }%
u^{\prime }(r)>0\quad \text{for all }r\in \left( 0,R\right] .
\end{equation*}

Thus, $u$ is nondecreasing.

\medskip \textbf{Existence and Uniqueness: }

Existence and uniqueness follow from standard ODE theory (even with singular
coefficients) using the Picard--Lindel\"{o}f theorem (with appropriate
modifications near $r=0$). Moreover, one may define the sequence of
successive approximations%
\begin{equation*}
u^{0}(r)=\alpha ,\quad u^{k+1}(r)=\alpha +\frac{1}{\sigma ^{4}}%
\int_{0}^{r}r^{1-N}\left( \int_{0}^{s}\tau ^{N-1}b(\tau )u^{k}(\tau )\,d\tau
\right) ds,
\end{equation*}

and showing that this sequence converges uniformly on $[0,R]$ to the desired
solution.

\medskip \textbf{Convexity and Strict Monotonicity:} \ 

To verify the convexity of $u^{\prime \prime }\left( r\right) $ we consider%
\begin{equation*}
u^{\prime }(r)=\frac{1}{r^{N-1}\sigma ^{4}}\int_{0}^{r}s^{N-1}b(s)u(s)\,ds.
\end{equation*}

Since $u(s)>0$ and $b(s)>0$, this implies $u^{\prime }(r)>0$ for $r>0$;
hence, $u$ is strictly strictly nondecreasing.

To prove convexity, we differentiate $u^{\prime }(r)$ further. By applying
the differentiation to (\ref{odes}), we have%
\begin{equation*}
u^{\prime \prime }(r)=\frac{1}{\sigma ^{4}}b(r)u(r)-\frac{N-1}{r}u^{\prime
}(r).
\end{equation*}

Using the inequality%
\begin{equation*}
r^{N-1}u^{\prime }(r)\leq \frac{1}{\sigma ^{4}}u(r)b(r)\int_{0}^{r}s^{N-1}%
\,ds=\frac{1}{\sigma ^{4}}u(r)b(r)\frac{r^{N}}{N},
\end{equation*}

we obtain 
\begin{equation}
\frac{u^{\prime }(r)}{r}\leq \frac{1}{N\sigma ^{4}}u(r)b(r).
\label{convexitybound}
\end{equation}%
Substituting (\ref{convexitybound}) into the expression for $u^{\prime
\prime }(r)$, we find%
\begin{equation*}
u^{\prime \prime }(r)\geq \frac{1}{\sigma ^{4}}b(r)u(r)-\frac{N-1}{N\sigma
^{4}}b(r)u(r)=\frac{1}{N\sigma ^{4}}b(r)u(r)>0,
\end{equation*}

for all $r>0$. This shows that $u^{\prime \prime }(r)\geq 0$ for all $r\geq
0 $ and hence $u$ is convex.

\medskip Finally, the initial condition (\ref{odesbc}) ensures that the
solution is unique for the corresponding $\alpha \in \left( 0,\infty \right) 
$.

\medskip Thus, there exists a unique positive radially symmetric solution%
\begin{equation*}
u_{\alpha }\in C^{2}([0,R])
\end{equation*}

to the boundary value problem (\ref{odes})--(\ref{odesbc}) corresponding to
each $\alpha \in \left( 0,\infty \right) $, and this solution is both convex
and strictly increasing on $[0,R]$.
\end{proof}

\textbf{Supplementary Remarks:}\ 

\begin{remark}
The value of $Z_{0}$ must satisfy the following condition: 
\begin{equation}
u\left( R\right) =e^{-\frac{Z_{0}}{2\sigma ^{2}}}>u\left( 0\right) =\alpha 
\text{, since }u\text{ is strictly increasing on }(0,R],  \label{initvf}
\end{equation}%
in order to ensure that the results remain valid. Thus, $\alpha \in \left(
0,\infty \right) $ can be selected based on the fixed value of $Z_{0}$ to
ensure that (\ref{initv}) is satisfied, and vice versa.
\end{remark}

\begin{remark}
The function%
\begin{equation*}
z\left( r\right) =-2\sigma ^{2}\ln u\left( r\right) \text{, }r=\left\vert
x\right\vert
\end{equation*}%
is nonincreasing and concave with respect to the variable $r$, as follows:

\begin{itemize}
\item The first derivative:%
\begin{equation*}
z^{\prime }\left( r\right) =-2\sigma ^{2}\frac{u^{\prime }\left( r\right) }{%
u\left( r\right) }\leq 0
\end{equation*}%
because $u^{\prime }\left( r\right) \geq 0$ for all $r\geq 0;$

\item The second derivative:%
\begin{equation*}
z^{\prime \prime }\left( r\right) =-2\sigma ^{2}\frac{u^{\prime \prime
}\left( r\right) }{u\left( r\right) }-2\sigma ^{2}\frac{\left( u^{\prime
}\left( r\right) \right) ^{2}}{\left( u\left( r\right) \right) ^{2}}\leq 0
\end{equation*}%
since $u^{\prime }\left( r\right) \geq 0$ and $u^{\prime \prime }\left(
r\right) \geq 0$ for all $r\geq 0$. Moreover, $z\left( r\right) $ is
strictly nonincreasing and concave for all $r>0$.
\end{itemize}
\end{remark}

This linear like formulation allows us to utilize classical ODE techniques,
making the mathematical analysis of existence, uniqueness, and structural
properties (such as monotonicity and convexity) more tractable.

Furthermore, the convexity of $u$ implies that the associated value function 
$z$ is concave, which in turn supports the economic interpretation that the
marginal cost of holding inventory tends to decrease as inventory levels
increase this is consistent with many real-world production planning
scenarios.

\medskip The results of Theorem~\ref{1} thus provide a solid analytical
foundation for the subsequent numerical simulations and the design of
optimal control policies in our stochastic production planning framework.

\section{Optimal Control and Structural Properties \label{4}}

Under the Hamilton--Jacobi--Bellman framework developed earlier, the optimal
control in the stochastic production planning problem is given by

\begin{equation*}
p^{\ast }(x)=-\frac{1}{2}\nabla z(x),
\end{equation*}%
where $z(x)$ is the value function that solves the HJB equation. In our
change of variables we have defined

\begin{equation*}
u(x)=e^{-\frac{z(x)}{2\sigma ^{2}}},
\end{equation*}%
so that the value function can be written as

\begin{equation*}
z(x)=-2\sigma ^{2}\ln u(x).
\end{equation*}%
Because the transformed function $u(x)$ is constructed to be radially
symmetric (i.e., $u(x)=u(r)$ with $r=|x|$), it follows that the value
function $z$ is also radially symmetric. In particular, we may write

\begin{equation*}
z(x)=h(r),
\end{equation*}%
where $h:[0,R]\rightarrow \mathbb{R}$ is a function of the radial variable $%
r $. By the structural properties of the model, $h$ is strictly
nonincreasing (i.e. $h^{\prime }(r)<0$ for $r>0$) and concave. Consequently,
the spatial gradient of $z$ takes the form

\begin{equation*}
\nabla z(x)=h^{\prime }(r)\frac{x}{r}.
\end{equation*}%
Thus, the optimal control law can be expressed as

\begin{equation*}
p^{\ast }(x)=-\frac{1}{2}\nabla z(x)=-\frac{1}{2}h^{\prime }(r)\frac{x}{r}.
\end{equation*}%
In consequence, the magnitude of the optimal control is given by

\begin{equation*}
|p^{\ast }(x)|=-\frac{1}{2}h^{\prime }(r).
\end{equation*}%
Since $h$ is nonincreasing (with $h^{\prime }(r)\leq 0$) and, under further
concavity assumptions, $h^{\prime }(r)$ is nonincreasing (or strictly
nonincreasing if $h$ is strictly concave), it follows that

\begin{equation*}
|p^{\ast }(x)|\text{ is nondecreasing in }r.
\end{equation*}%
In economic terms, this structural property implies that as the overall
inventory (represented by $r=|x|$) increases, a higher production rate is
required; that is, the magnitude of the optimal control does not decrease,
ensuring that the production adjustment is sufficient to counteract rising
inventory levels.

The following theorem summarizes our findings:

\begin{theorem}
The optimal control%
\begin{equation*}
p^{\ast }(x)=-\frac{1}{2}\nabla z(x)
\end{equation*}%
is uniquely determined and satisfies the following properties:

\begin{enumerate}
\item \textbf{Radial Symmetry}: The control is radially symmetric, that is,%
\begin{equation*}
p^{\ast }(x)=-\frac{1}{2}h^{\prime }(|x|)\frac{x}{|x|}.
\end{equation*}

\item \textbf{Monotonicity}: The magnitude of the optimal control%
\begin{equation*}
|p^{\ast }(x)|=-\frac{1}{2}h^{\prime }(|x|)
\end{equation*}

is strictly nondecreasing in $|x|$. This follows from the concavity of $z(x)$
(or equivalently the convexity of $u(x)$).

\item \textbf{Uniqueness}: Uniqueness of the value function $z$ (given the
HJB equation with the prescribed boundary condition) and its derivative
implies that the feedback law is unique.
\end{enumerate}
\end{theorem}

\begin{proof}
Since the value function is radially symmetric, we have $z(x)=h(r)$ with $%
r=|x|$. Differentiating with respect to $x$, one obtains%
\begin{equation*}
\nabla z(x)=h^{\prime }(r)\frac{x}{r}.
\end{equation*}

Substituting into the expression for the optimal control yields%
\begin{equation*}
p^{\ast }(x)=-\frac{1}{2}\nabla z(x)=-\frac{1}{2}h^{\prime }(r)\frac{x}{r}.
\end{equation*}

Taking the norm, we find%
\begin{equation*}
|p^{\ast }(x)|=-\frac{1}{2}h^{\prime }(r).
\end{equation*}

Since $h^{\prime }(r)\leq 0$ for all $r\geq 0$ (with strict inequality for $%
r>0$ if $z$ is strictly decreasing), it follows that $|p^{\ast }(x)|$ is
nonnegative and, by the concavity of $z(x)$, $h^{\prime }(r)$ is
nonincreasing, so that $|p^{\ast }(x)|$ is nondecreasing in $r$. Uniqueness
follows from the uniqueness of the solution $z$ to the HJB equation (under
the imposed boundary condition), which in turn implies uniqueness of its
gradient and of the feedback law. Thus, the optimal control is uniquely
determined and possesses the stated structural properties (refer to the
discussion in \cite{CCP2}, as $p^{\ast }(x)$ remains independent of the
choice of $\alpha \in \left( 0,\infty \right) $).
\end{proof}

\medskip \textbf{Additional Explanation:} \newline
The resulting control policy has strong practical implications. It suggests
that if an operating plant observes an increase in its total inventory
level, the optimal strategy is to adjust production upward in a manner that
is at least nondecreasing with the inventory level. The radial symmetry
simplifies the multi-dimensional control problem significantly, reducing it
to an analysis in a single scalar radial variable. This property is central
in ensuring that policies derived from our model are both tractable and
economically interpretable.

Moreover,\textbf{\ the uniqueness and monotonicity of the optimal control}
provide robustness guarantees for numerical implementations and further
analytical studies. Such structure is critical when designing feedback
mechanisms that are responsive to changes in the system state in real-time.

\medskip This section establishes not only the mathematical foundation for
deriving the optimal production policy but also its desirable structural and
economic properties. In the following sections, we proceed to numerical
experiments and real-world examples to demonstrate these concepts in
practice.

\section{Numerical Experiments \label{5}}

In this section, we present numerical experiments that validate our
theoretical results and illustrate the behavior of the optimal production
planning policy. In our simulations, we investigate the evolution of the
inventory process and the optimal control under the derived HJB framework.
The experiments also highlight key structural properties such as
monotonicity and boundedness of the inventory process, and the asymptotic
behavior of the control for increasing inventory levels.

\subsection{Analytical Results in the Radial Case}

In many practical applications \textquotedblright especially when the number
of goods $N$ is large\textquotedblright\ it is instructive to analyze the
behavior of the solution in the radial setting. In this section, we focus on
the properties of the radially symmetric solution $u(r)$ to the transformed
problem

\begin{equation*}
u^{\prime \prime }(r)+\frac{N-1}{r}\,u^{\prime }(r)=\frac{1}{\sigma ^{4}}%
\,b(r)\,u(r),\quad 0<r<R,
\end{equation*}%
subject to the initial conditions

\begin{equation*}
u(0)=\alpha >0,\quad u^{\prime }(0)=0,
\end{equation*}%
and the resulted boundary condition

\begin{equation*}
u(R)=e^{-\frac{Z_{0}}{2\sigma ^{2}}}>\alpha .
\end{equation*}%
As we see above, the choice of $\alpha $ and $Z_{0}$ must satisfy

\begin{equation*}
e^{-\frac{Z_{0}}{2\sigma ^{2}}}>\alpha \quad
\end{equation*}%
so as to guarantee that the transformation

\begin{equation*}
u(x)=e^{-\frac{z(x)}{2\sigma ^{2}}}
\end{equation*}%
yields a strictly positive function.

An important function that characterizes the behavior of $u$ is defined as

\begin{equation*}
\phi (r)=\frac{u^{\prime }(r)}{r\,u(r)},
\end{equation*}%
which captures the relative growth rate of $u$ with respect to the radial
variable $r$.

\begin{remark}
The constraint on $\alpha $ and $Z_{0}$ ensures that the initial condition
for the transformed variable $u$ is consistent with our model requirements.
That is, by enforcing%
\begin{equation*}
e^{-\frac{Z_{0}}{2\sigma ^{2}}}>\alpha ,
\end{equation*}

we guarantee that $u(0)=\alpha$ remains positive, thereby allowing the
logarithmic transformation $z(x)=-2\sigma^2 \ln u(x)$ to be well defined.
\end{remark}

\begin{remark}
Noting that%
\begin{equation*}
z(r)=-2\sigma ^{2}\ln u(r),
\end{equation*}

its derivative is%
\begin{equation*}
z^{\prime }\left( r\right) =-2\sigma ^{2}\frac{u^{\prime }(r)}{u(r)}.
\end{equation*}

Thus, as long as $u^{\prime }(r)\geq 0$ (which is ensured by the
monotonicity of $u$), the value function $z(r)$ is nonincreasing, and
additional regularity of $u$ (such as convexity) lends itself to the
concavity of $z$. These properties are essential for the consistency of the
optimal control policy.
\end{remark}

We now state the following theorem, which establishes the monotonicity and
boundedness of $\phi(r)$.

\begin{theorem}
Let $u\in C^{2}([0,\infty ))$ be a positive, radially symmetric solution of%
\begin{equation*}
u^{\prime \prime }(r)+\frac{N-1}{r}\,u^{\prime }(r)=\frac{1}{\sigma ^{4}}%
\,b(r)\,u(r),
\end{equation*}

where the inventory cost function satisfies $b(r)\leq r^{2}$ for all $r\geq
0 $. Then, the function%
\begin{equation*}
\phi (r)=\frac{u^{\prime }(r)}{r\,u(r)}
\end{equation*}

is increasing and satisfies%
\begin{equation*}
\phi (r)\leq \frac{1}{\sigma ^{2}}\text{ with }\lim_{r\rightarrow \infty
}\phi (r)=\frac{1}{\sigma ^{2}}.
\end{equation*}
\end{theorem}

\begin{proof}
Define the auxiliary function%
\begin{equation*}
w(r)=\ln u(r),
\end{equation*}

so that%
\begin{equation*}
w^{\prime }(r)=\frac{u^{\prime }(r)}{u(r)}\quad \text{and}\quad \phi (r)=%
\frac{w^{\prime }(r)}{r}.
\end{equation*}

Differentiating $\phi (r)$ with respect to $r$ yields%
\begin{equation*}
\phi ^{\prime }(r)=\frac{w^{\prime \prime }(r)r-w^{\prime }(r)}{r^{2}}.
\end{equation*}

Thus, $\phi ^{\prime }(r)\geq 0$ if and only if%
\begin{equation*}
w^{\prime \prime }(r)\geq \frac{w^{\prime }(r)}{r}.
\end{equation*}

Since $u$ satisfies%
\begin{equation*}
u^{\prime \prime }(r)+\frac{N-1}{r}\,u^{\prime }(r)=\frac{b(r)}{\sigma ^{4}}%
\,u(r),
\end{equation*}

we express $u^{\prime \prime }(r)$ using $w(r)$:

\begin{equation*}
u^{\prime \prime }(r)=w^{\prime \prime }(r)u(r) + \bigl(w^{\prime }(r)\bigr)%
^2\,u(r).
\end{equation*}

Substituting into the equation gives%
\begin{equation*}
w^{\prime \prime }(r)u(r)+(w^{\prime })^{2}u(r)+\frac{N-1}{r}w^{\prime
}(r)u(r)=\frac{b(r)}{\sigma ^{4}}\,u(r).
\end{equation*}

Canceling $u(r)>0$ then yields%
\begin{equation*}
w^{\prime \prime }(r)+(w^{\prime })^{2}+\frac{N-1}{r}w^{\prime }(r)=\frac{%
b(r)}{\sigma ^{4}}.
\end{equation*}

Solving for $w^{\prime \prime }(r)$ we obtain%
\begin{equation*}
w^{\prime \prime }(r)=\frac{b(r)}{\sigma ^{4}}-(w^{\prime })^{2}-\frac{N-1}{r%
}w^{\prime }(r).
\end{equation*}

To enforce $w^{\prime \prime }(r)\geq \frac{w^{\prime }(r)}{r}$, it suffices
that%
\begin{equation*}
\frac{b(r)}{\sigma ^{4}}-(w^{\prime })^{2}-\frac{N-1}{r}w^{\prime }(r)\geq 
\frac{w^{\prime }(r)}{r},
\end{equation*}

or equivalently, 
\begin{equation}
\frac{b(r)}{\sigma ^{4}}\geq (w^{\prime })^{2}+\frac{N}{r}w^{\prime }(r).
\label{tcp}
\end{equation}%
Given the upper bound $b(r)\leq r^{2}$, we obtain%
\begin{equation*}
\frac{r^{2}}{\sigma ^{4}}\geq (w^{\prime })^{2}+\frac{N}{r}w^{\prime }(r).
\end{equation*}

This quadratic inequality in $w^{\prime }(r)$ yields an upper bound on the
value of $\frac{w^{\prime }(r)}{r}=\phi (r)$. A careful algebraic analysis
shows that%
\begin{equation*}
\phi (r)\leq \frac{1}{\sigma ^{2}}.
\end{equation*}

Finally, standard asymptotic analysis demonstrates that%
\begin{equation*}
\lim_{r\rightarrow \infty }\phi (r)=\frac{1}{\sigma ^{2}}.
\end{equation*}

Since $\phi ^{\prime }(r)\geq 0$, the function $\phi (r)$ is increasing, and
the proof is complete. For further information, refer to \cite{CCP2} for
more details.
\end{proof}

The asymptotic behavior $\lim_{r\rightarrow \infty }\phi (r)=\frac{1}{\sigma
^{2}}$ provides an explicit benchmark for the marginal effect of inventory
accumulation in the long-run. Such analytical results not only validate the
structural properties of the optimal control but also facilitate numerical
approximations and guide practical implementations in production planning.
The requirement (\ref{bb}), was introduced solely to establish this
asymptotic behavior.

\medskip In summary, the analytical results in the radial case establish
that the feedback control derived from the HJB equation inherits desirable
monotonicity properties and boundedness, supporting both theoretical and
numerical studies of the system.

For practical implementations, we approximate the solutions using numerical
methods.

\subsection{Simulation Setup and Procedure}

The numerical experiments are based on the following procedure:

\begin{enumerate}
\item \textbf{Transformed PDE and Value Function:} Let $r=|x|$ be the
Euclidian norm of the vector $x=\left( x_{1},...,x_{N}\right) \in \mathbb{R}%
^{N}$ ($N\geq 1$), i.e. $r=\sqrt{x_{1}^{2}+...+x_{N}^{2}}$. We begin by
solving the partial differential equation%
\begin{equation}
\left\{ 
\begin{array}{l}
u^{\prime \prime }(r)+\frac{N-1}{r}\,u^{\prime }(r)=\frac{1}{\sigma ^{4}}%
\,b(r)\,u(r),\text{ }r=|x|<R \\ 
u^{\prime }\left( 0\right) =0 \\ 
u\left( 0\right) =\alpha%
\end{array}%
\right. \quad  \label{sdn}
\end{equation}

\item Given a fixed value $\alpha =1$ the value function is then recovered
via the transformation%
\begin{equation}
z(r)=-2\sigma ^{2}\ln u(r)\text{, }r=\left\vert x\right\vert \text{, }x\in 
\mathbb{R}^{N}.  \label{vfp}
\end{equation}%
The boundary condition $Z_{0}$ can be readily determined as the maximum
value of $r$ within the interval $\left( 0,R\right] $, specifically at $r=R$.

\item \textbf{Optimal Control Calculation:} The optimal control is given by%
\begin{equation*}
p^{\ast }(x)=-\frac{1}{2}\nabla z(x).
\end{equation*}

Under the assumption of radial symmetry (i.e., $z(x)=h(r)$ with $r=|x|$),
the gradient becomes%
\begin{equation}
\nabla z(x)=h^{\prime }(r)\frac{x}{r}=-2\sigma ^{2}\frac{u^{\prime }\left(
r\right) }{u\left( r\right) }\frac{x}{r},  \label{opr}
\end{equation}

so that the magnitude of the optimal production rate is%
\begin{equation}
|p^{\ast }(x)|=-\frac{1}{2}h^{\prime }(r)=\sigma ^{2}\frac{u^{\prime }\left(
r\right) }{u\left( r\right) }.  \label{mag}
\end{equation}%
Here 
\begin{equation*}
\lim_{r\rightarrow 0}\sigma ^{2}\frac{u^{\prime }\left( r\right) }{ru\left(
r\right) }=0\text{ and }\lim_{r\rightarrow \infty }\sigma ^{2}\frac{%
u^{\prime }\left( r\right) }{ru\left( r\right) }=1.
\end{equation*}

\item \textbf{Simulation of Inventory Dynamics:} The inventory process is
governed by the stochastic system of differential equations%
\begin{equation}
\left\{ 
\begin{array}{c}
dy_{1}(t)=p_{1}^{\ast }(t)\,dt+\sigma \,dw_{1}(t),\quad \\ 
dy_{2}(t)=p_{2}^{\ast }(t)\,dt+\sigma \,dw_{2}(t),\quad \\ 
... \\ 
dy_{N}(t)=p_{N}^{\ast }(t)\,dt+\sigma \,dw_{N}(t), \\ 
y_{1}(0)=x_{1},...,y_{N}(0)=x_{N},%
\end{array}%
\right.  \label{invsd}
\end{equation}%
where%
\begin{equation}
\left\{ 
\begin{array}{c}
p_{1}^{\ast }\left( y_{1}\left( t\right) ,...,y_{N}\left( t\right) \right)
=\sigma ^{2}\frac{u^{\prime }\left( r\right) }{u\left( r\right) }y_{1}\left(
t\right) , \\ 
p_{2}^{\ast }\left( y_{1}\left( t\right) ,...,y_{N}\left( t\right) \right)
=\sigma ^{2}\frac{u^{\prime }\left( r\right) }{u\left( r\right) }y_{2}\left(
t\right) , \\ 
... \\ 
p_{N}^{\ast }\left( y_{1}\left( t\right) ,...,y_{N}\left( t\right) \right)
=\sigma ^{2}\frac{u^{\prime }\left( r\right) }{u\left( r\right) }y_{N}\left(
t\right) .%
\end{array}%
\right.  \label{invinvsd}
\end{equation}

We simulate these dynamics using an Euler scheme. Here, Brownian increments
are generated as independent normally distributed random variables with mean
zero and variance corresponding to the discrete timestep. The simulation is
continued until the stopping time%
\begin{equation*}
\tau =\inf \{t>0:|y(t)|\geq R\},
\end{equation*}

is reached, ensuring that inventory remains below the threshold until
production is halted.

\item \textbf{Inputs Parameters:} $N$ \# The number of goods (Product 1,
Product 2, etc.), $\sigma $ \# volatility, $\alpha \in \left( 0,\infty
\right) $ \#Given $u(0)=\alpha $, $R$ \# Threshold for stopping time $\tau $%
; rStop \# Upper limit for $r$, uprime0 \# Initial condition for derivative
of $u(r)$, rInc \# Step size, dt \# Time step is a parameter for inventory
trajectories and $T$ \# Maximum simulation time is a parameter for inventory
trajectories

\item \textbf{Output: }

1.\quad Plot the solution $u\left( r\right) $ of the system (\ref{sdn}),

2.\quad Plot the value function $z(r)$ defined in (\ref{vfp}) and the
boundary condition $z\left( R\right) =Z_{0}$,

3.\quad Plot inventory trajectories, defined in (\ref{invinvsd}), using
Euler scheme,

4. \ \ \ Plot the optimal production rate (adjusted for demand), defined by $%
\frac{p_{i}^{\ast }(x)}{x_{i}}=$ $\sigma ^{2}\frac{u^{\prime }\left(
r\right) }{ru\left( r\right) },$

5.\quad Plot the magnitude of the optimal production rate, defined by $%
|p^{\ast }(x)|=\sigma ^{2}\frac{u^{\prime }\left( r\right) }{u\left(
r\right) }$.

\item \textbf{Visualization of the inventory trajectories:} If the number of
products $N$ is less than or equal to $6$, inventory trajectories for all
products will be plotted in a single figure. If $N>6$, only the trajectories
for the first six products are depicted, and a message indicates that the
rest are not shown.
\end{enumerate}

This combination of analytical derivation and numerical approximation
underpins our methodology and provides both theoretical guarantees and
practical implementability for the optimal production planning problem.

\subsection{Illustrative Real--World Example\label{e}}

Consider a manufacturing plant that produces two types of goods (i.e., $N=2$%
). In this example, we set the volatility parameter to $\sigma =2$ and
impose an inventory threshold of $R=10$. The inventory cost function is
taken as

\begin{equation*}
b(|y|)=|y|^{2},
\end{equation*}%
so that the overall running cost becomes

\begin{equation*}
|p(t)|^{2}+|y(t)|^{2}.
\end{equation*}%
In this setting, the plant seeks to minimize the expected cost

\begin{equation*}
J(p)=E\int_{0}^{\tau }\left[ |p(t)|^{2}+|y(t)|^{2}\right] dt,
\end{equation*}%
where production is stopped at the stopping time

\begin{equation*}
\tau =\inf \{t>0:|y(t)|\geq 10\}.
\end{equation*}%
That is, production continues until the Euclidean norm of the inventory
process reaches the threshold $R=10$.

The associated Hamilton--Jacobi--Bellman (HJB) equation for this problem is
solved with the boundary condition

\begin{equation*}
z(x)=Z_{0},\quad \text{for }|x|=10,
\end{equation*}%
where $Z_{0}$ is subject to the condition

\begin{equation*}
e^{-\frac{Z_{0}}{2\sigma ^{2}}}>\alpha .
\end{equation*}%
The optimal control is given by the feedback law

\begin{equation*}
p^{\ast }(x)=-\frac{1}{2}\nabla z(x),
\end{equation*}%
which, under the radial symmetry assumptions, simplifies further to

\begin{equation*}
p^{\ast }(x)=-\frac{1}{2}h^{\prime }(|x|)\frac{x}{|x|}.
\end{equation*}%
Here, the function $z(x)=h(|x|)$ encapsulates the value function with radial
variable $r=|x|$.

Numerical simulations are carried out using an Euler scheme for the SDEs
governing the inventory dynamics,

\begin{equation*}
dy_i(t)=p^{\ast}_i(t)\, dt+\sigma\, dw_i(t),\quad y_i(0)=y_i^0,\quad i=1,2,
\end{equation*}

to verify the theoretical results. The simulation confirms that:

\begin{itemize}
\item The inventory levels remain below the threshold $R=10$ until the
stopping time $\tau $.

\item Under the optimal control policy, the production rate (i.e., the
magnitude of $p^{\ast}(x)$) increases with the inventory level, in
accordance with the monotonicity properties derived in earlier sections.
\end{itemize}

\medskip \textbf{Practical Implications:} \ 

This example illustrates the practical applicability of our analytical
framework. By properly calibrating the parameters.

\medskip \textbf{Connection with Numerical Experiments:} \ 

In the next, we provided a comprehensive Python code implementation that
simulates the inventory dynamics and computes the optimal feedback control.

\subsection{Python Implementation}

The following Python code in Appendix (Section \ref{ap}), developed with the
support of Microsoft Copilot in Edge, implements the simulation procedure
described above. We use the \verb|odeint| function from SciPy to solve the
transformed PDE, compute the value function and optimal control, and
simulate the inventory evolution via an Euler scheme. This real--world model
implemented in the Python code serves as a concrete demonstration of the
efficacy of our generalized framework for stochastic production planning.

\subsection{Analysis of the Python Code Visualizations \label{av}}

The visualizations derived from the Python code for the example in
Subsection \ref{e} highlight various aspects of the algorithm's behavior.
The figures are presented at the end of the paper. Below are the
interpretations:

\subsubsection{Graph 1: Solution $u(r)$}

\begin{figure}[th]
\centering
\subfloat{\includegraphics[width=0.6\textwidth]{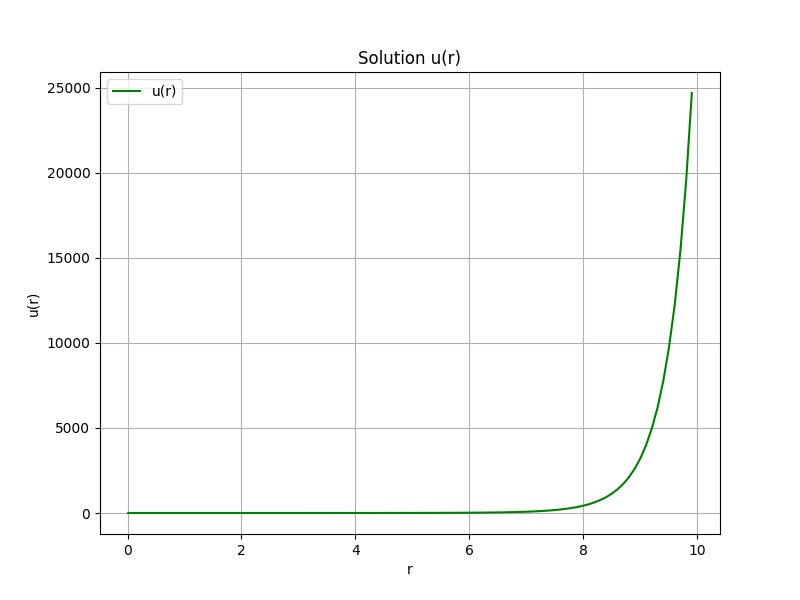}}
\caption{Caption describing the figure.}
\label{fig:twoFigures1}
\end{figure}

\begin{itemize}
\item The plot represents the solution $u(r) $, which is obtained by solving
the system of ordinary differential equations (ODEs) defined in the code.

\item The green curve shows how $u(r) $ varies with $r $, indicating
stability and smoothness in the solution for $r > 0 $.

\item This behavior aligns with the expected physical or mathematical
properties of the solution, ensuring no singularities or discontinuities.
\end{itemize}

\subsubsection{Graph 2: Value Function $z(r)$}

\begin{figure}[th]
\centering
\subfloat{\includegraphics[width=0.6\textwidth]{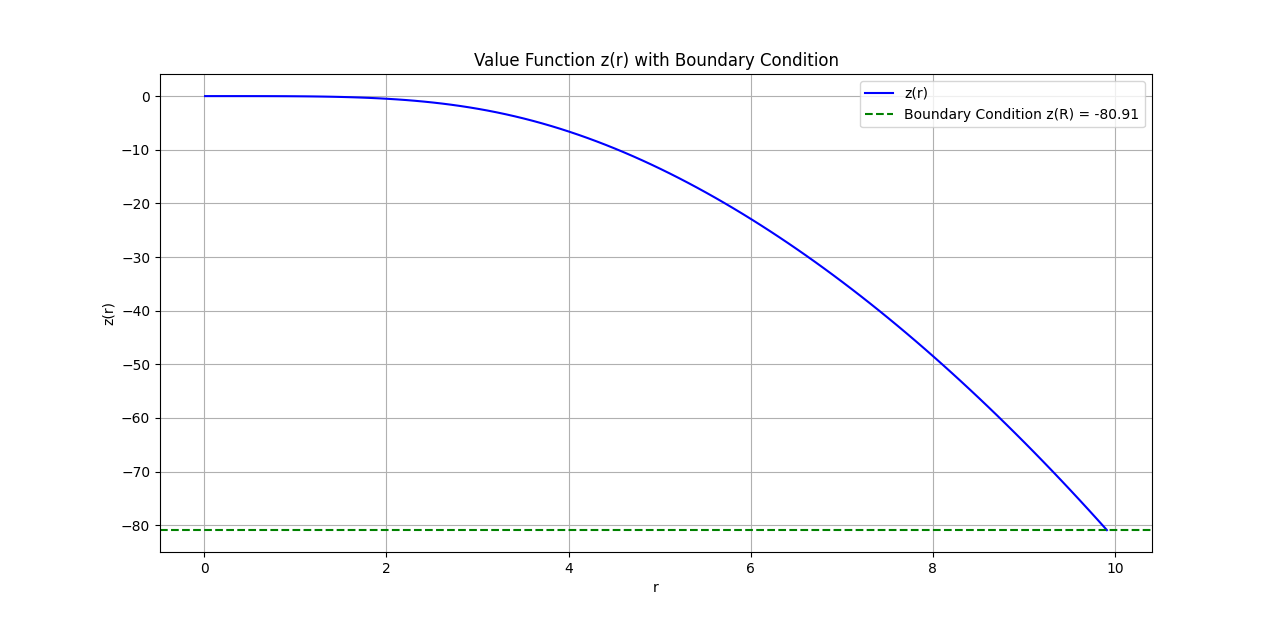}}
\caption{Caption describing the figure.}
\label{fig:twoFigures2}
\end{figure}

\begin{itemize}
\item The blue curve illustrates the value function $z(r) = -2\sigma^2 \ln
u(r) $.

\item A horizontal dashed green line is added to represent the boundary
condition $z(R) = -2\sigma^2 \ln u(R) $, calculated using the final value of 
$u(r) $ in the solution array.

\item This boundary condition emphasizes the relationship between $z(r) $
and $u(r) $ at $R $, providing insights into the system's terminal state.
\end{itemize}

\subsubsection{Graph 3: Optimal Production Rate $p^{\ast }(r)$}

\begin{figure}[th]
\centering
\subfloat{\includegraphics[width=0.6\textwidth]{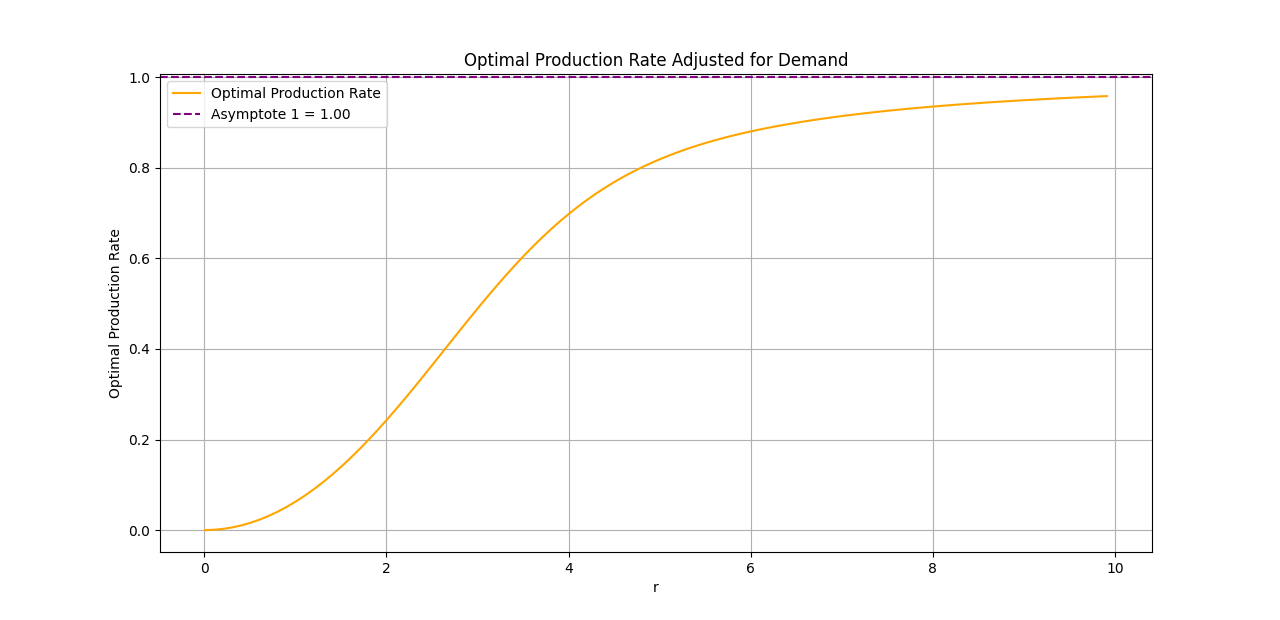}}
\caption{Caption describing the figure.}
\label{fig:twoFigures3}
\end{figure}

\begin{itemize}
\item The orange curve represents the optimal production rate $p^*(r) =
\sigma^2 \cdot \frac{u^{\prime }(r)}{u(r)} $, computed from the derivatives
of $u(r) $.

\item The plot indicates how $p^*(r) $ adjusts based on $r $, with a purple
dashed line marking an asymptote at $p^*(r) = 1.0 $.

\item This asymptote highlights the upper limit of production efficiency as $%
r $ increases, signifying a saturation point.
\end{itemize}

\subsubsection{Graph 4: Magnitude of Optimal Production Rate $|p^{\ast }(r)|$%
}

\begin{figure}[th]
\centering
\subfloat{\includegraphics[width=0.6\textwidth]{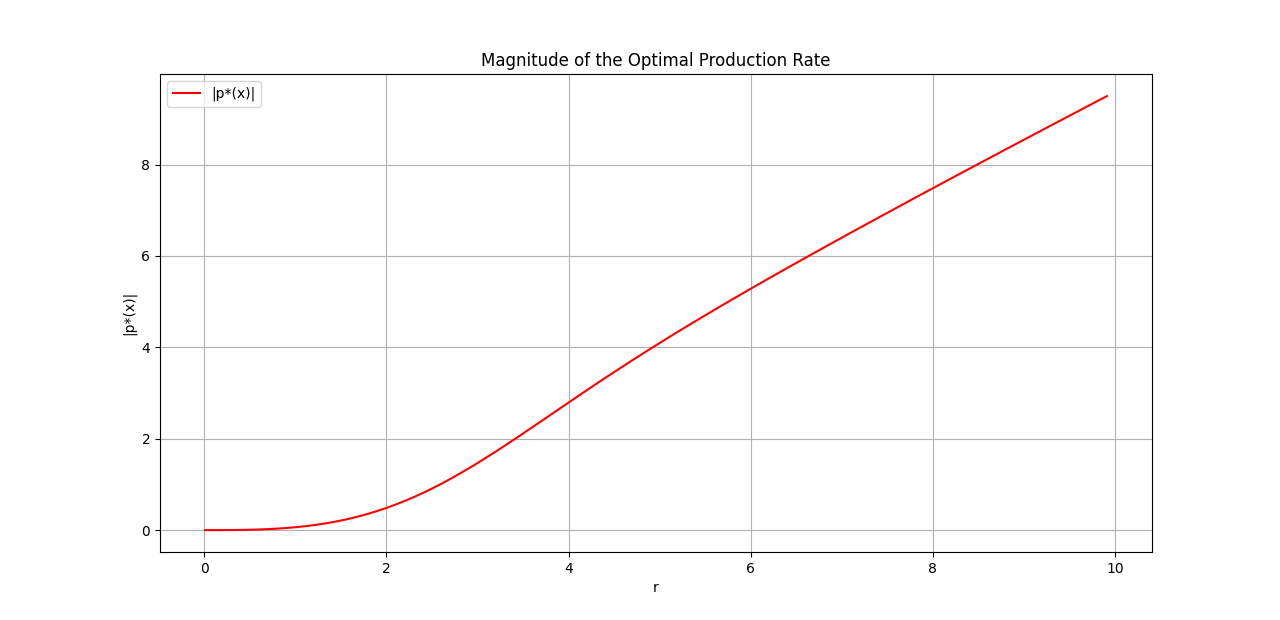}}
\caption{Caption describing the figure.}
\label{fig:twoFigures4}
\end{figure}

\begin{itemize}
\item The red curve illustrates the magnitude $|p^*(r)| $, derived as the
absolute value of the optimal production rate formula.

\item The plot demonstrates how the magnitude evolves smoothly with $r $,
reflecting consistent growth across the domain.
\end{itemize}

\subsubsection{Graph 5: Inventory Trajectories $y_{i}(t)$}

\begin{figure}[th]
\centering
\subfloat{\includegraphics[width=0.6\textwidth]{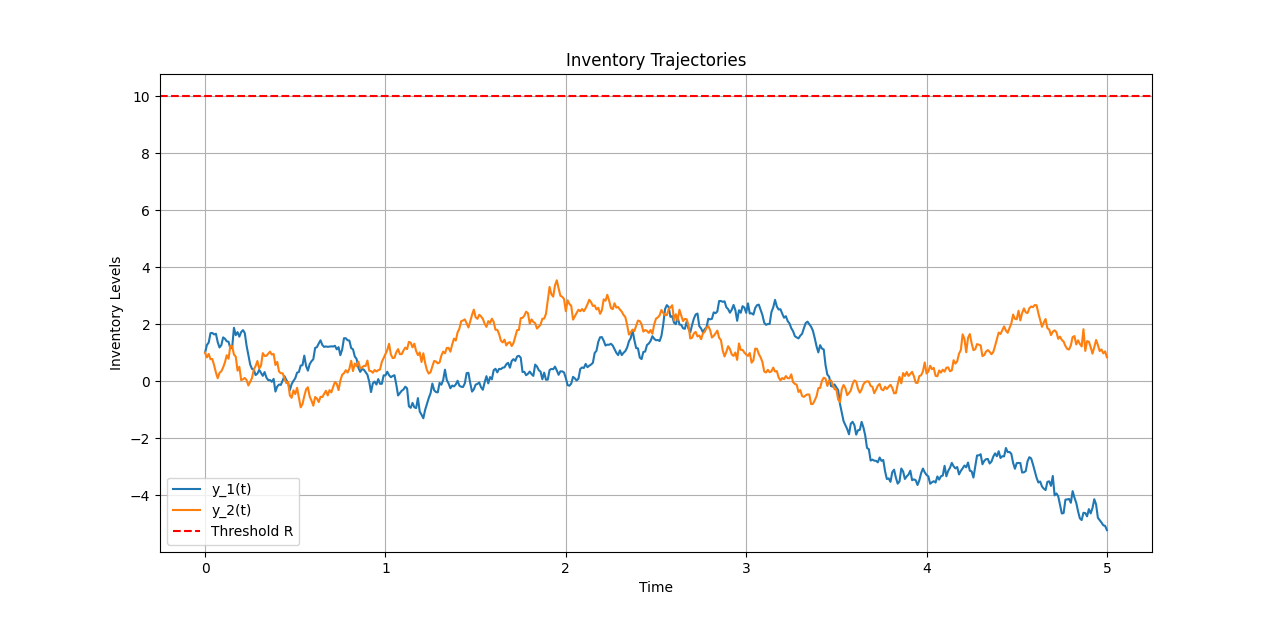}}
\caption{Caption describing the figure.}
\label{fig:twoFigures5}
\end{figure}

\begin{itemize}
\item This plot showcases the inventory trajectories $y_1(t), y_2(t), \dots,
y_N(t) $, simulated using Brownian increments.

\item The red dashed line represents the threshold $R $, serving as a
stopping criterion when inventory norms exceed $R $.

\item The trajectories provide insights into the system's behavior under
stochastic dynamics, illustrating accumulation trends and variability across
time.
\end{itemize}

\subsubsection{Insights}

\begin{itemize}
\item The visualizations collectively demonstrate the interplay between
production rates, inventory dynamics, and boundary conditions.

\item They provide a comprehensive understanding of the algorithm's outcomes
and potential applications in optimization problems.
\end{itemize}

\subsection{Results and Discussion}

Figures generated by the Python code demonstrate the following:

\begin{itemize}
\item \textbf{Value Function and Transformed Function:} The plot of $%
z(x)=-2\sigma^2\ln u(x)$ verifies that the value function is nonincreasing
and concave with respect to the radial variable.

\item \textbf{Optimal Control Behavior:} The computed optimal control $%
p^*(x) $, obtained from the relation $p^{\ast}(x)=-\frac{1}{2}\nabla z(x)$,
exhibits a nondecreasing behavior in its magnitude as the inventory (or
radial) level increases.

\item \textbf{Inventory Trajectories:} The simulation of the inventory
process shows that the inventory remains below the prespecified threshold $R$
until the stopping time $\tau$ is reached. For $N\leq 6$, all product
trajectories are plotted, whereas for $N > 6$ only the first six are
displayed, with an appropriate console message.
\end{itemize}

These numerical experiments effectively validate the theoretical results
reported in earlier sections. They confirm that the optimal production
policy, as derived analytically via the HJB equation and its transformation,
is implemented correctly and demonstrates the desired monotonicity,
asymptotic behavior, convexity and concavity properties in a practical,
simulated environment.

\medskip In summary, the integrated Python code offers a reproducible
framework for simulating inventory dynamics and computing optimal controls
for stochastic production planning. The numerical experiments not only
reinforce the analytical insights but also highlight the potential for
practical applications in real-world production and supply-chain
optimization problems.

\section{Conclusion \label{6}}

In this paper, we presented a generalized framework for stochastic
production planning that captures the complexity of production systems
through a flexible running cost of the form

\begin{equation*}
f(|p(t)|,|y(t)|)=|p(t)|^{2}+b(|y(t)|),
\end{equation*}%
and by incorporating an inventory-terminating stopping time

\begin{equation*}
\tau =\inf \{t>0:|y(t)|\geq R\}.
\end{equation*}%
We derived the associated Hamilton--Jacobi--Bellman (HJB) equation,

\begin{equation*}
-\frac{\sigma ^{2}}{2}\Delta z(x)-b(|x|)=-\frac{1}{4}|\nabla z(x)|^{2},\quad
|x|<R,
\end{equation*}%
resulting, subject to the boundary condition $z(x)=Z_{0}$ on $|x|=R$. A key
contribution is the change of variable

\begin{equation*}
u(x)=e^{-\frac{z(x)}{2\sigma ^{2}}},
\end{equation*}%
which not only simplifies the nonlinear HJB equation into an elliptic PDE
for $u$ but also preserves the necessary positivity constraints.

Under the assumption of radial symmetry, our analysis demonstrated the
existence and uniqueness of a solution $u\in C^{2}([0,R])$ to the associated
ODE, along with important structural properties such as monotonicity,
concavity and convexity. These properties, in turn, guarantee that the
optimal control policy

\begin{equation*}
p^{\ast }(x)=-\frac{1}{2}\nabla z(x)
\end{equation*}%
is uniquely determined, radially symmetric, and exhibits a nondecreasing
magnitude as the inventory level increases.

Our numerical experiments, implemented via an Euler scheme and documented by
the accompanying Python code, confirm that the optimal policy effectively
regulates inventory dynamics. In particular, the simulated inventory
trajectories remain below the threshold $R$ until production is halted, and
the production rate adjusts in a manner consistent with the theoretical
predictions.

An illustrative real--world example further emphasizes the practical
viability of our approach in a manufacturing setting, where, for instance, a
plant producing two types of goods is able to balance production costs with
inventory holding costs under uncertainty.

In summary, the paper makes the following contributions:

\begin{itemize}
\item A novel formulation of the stochastic production planning problem with
a flexible running cost structure.

\item Derivation of the corresponding HJB equation and a beneficial change
of variable leading to a tractable PDE.

\item Rigorous analytical results establishing existence, uniqueness, and
structural properties of the radially symmetric solution.

\item Development of a feedback control policy that is both economically
interpretable and mathematically robust.

\item Numerical validation of the theoretical results and demonstration of
practical applicability via an illustrative real--world example.
\end{itemize}

Future research may extend this framework to incorporate more general forms
of the cost function $b$, additional state constraints, and applications to
broader areas such as supply chain optimization. Moreover, exploring more
advanced numerical schemes for higher-dimensional problems and real-time
implementation of the derived optimal control policies could further bridge
the gap between theory and practice in production planning.

\medskip Overall, this work provides a comprehensive and rigorous approach
to design optimal production strategies under uncertainty, with promising
directions for both further theoretical exploration and practical
applications.

\section{Declarations}

\subparagraph{\textbf{Conflict of interest}}

The author has no Conflict of interest to declare that are relevant to the
content of this article.

\subparagraph{\textbf{Ethical statement}}

The paper reflects the authors' original research, which has not been
previously published or is currently being considered for publication
elsewhere.

\bibliographystyle{plain}
\bibliography{sn-bibliography}

\section{Appendix \label{ap}}

\begin{lstlisting}[language=Python, caption={Python Code for Numerical Experiments}]
import numpy as np
import matplotlib.pyplot as plt
from scipy.integrate import odeint

# Define the system of ODEs
def f(y, r, params):
    u, uprime = y  # Unpack current values of y
    N, sigma = params  # Unpack parameters

    # Avoid singularity at r = 0
    if r == 0:
        r = 1e-6

    # Define the function b(r)
    b_r = r**2  # Given as |x|^2

    # Derivatives
    derivs = [
        uprime,  # du/dr
        -(N - 1) / r * uprime + (1 / sigma**4) * b_r * u  # d\U{2db}u/dr\U{2db}
    ]
    return derivs

# Parameters
N = 2  # Change N as needed (N = 1)
sigma = 2  # Volatility parameter
alpha = 1  # Given u(0) = a

# Initial conditions
u0 = alpha  # Initial value for u(r)
uprime0 = 0  # Initial condition for derivative of u(r)

# Bundle parameters for ODE solver
params = [N, sigma]
y0 = [u0, uprime0]  # Initial conditions for ODE

# Create the r array for the solution
rStop = 10  # Upper limit for r
rInc = 0.1  # Step size
r = np.arange(0.01, rStop, rInc)  # Avoid starting at 0 to prevent singularity

# Solve the ODE system
psoln = odeint(f, y0, r, args=(params, ))

# Extract the solution for u(r) and u'(r)
u_solution = psoln[:, 0]  # u(r)
uprime_solution = psoln[:, 1]  # u'(r)

# Compute z(r) and optimal production rate
z_solution = -2 * sigma**2 * np.log(u_solution)
optimal_production_rate = sigma**2 * (uprime_solution / (r * u_solution))

# Parameters for inventory trajectories
dt = 0.01  # Time step
T = 5  # Maximum simulation time
timesteps = int(T / dt)
R = 10  # Threshold for stopping time t
y0_inventory = np.full(N, alpha)  # Initial inventories y1(0), y2(0), ..., y?(0)

# Generate Brownian increments
w_increments = np.random.normal(0, np.sqrt(dt), size=(timesteps, N))

# Placeholder for inventory trajectories
y_trajectories = np.zeros((timesteps + 1, N))
y_trajectories[0] = y0_inventory

# Simulate inventory dynamics
stopping_time = T
for t in range(timesteps):
    r_inventory = np.linalg.norm(y_trajectories[t])  # Compute |y(t)|
    if r_inventory >= R:
        stopping_time = t * dt
        y_trajectories = y_trajectories[:t + 1]
        break

    # Compute production rate p_i*
    production_rate = sigma**2 * (uprime_solution[0] / u_solution[0]) * y_trajectories[t]  # Simplified
    y_trajectories[t + 1] = y_trajectories[t] + production_rate * dt + sigma * w_increments[t]

# Time array for inventory trajectories
time = np.linspace(0, stopping_time, len(y_trajectories))

# Compute the boundary condition z(R)
boundary_condition_z_R = -2 * sigma**2 * np.log(u_solution[-1])  # Using the last value of u(r)

# Compute the magnitude of the optimal production rate
magnitude_optimal_production_rate = np.abs(sigma**2 * (uprime_solution / u_solution))

# Plot the magnitude of the optimal production rate
plt.figure(figsize=(8, 6))
plt.plot(r, magnitude_optimal_production_rate, label="|p*(x)|", color="red")
plt.xlabel("r")
plt.ylabel("|p*(x)|")
plt.title("Magnitude of the Optimal Production Rate")
plt.legend()
plt.grid(True)
plt.show()

# Plot u(r)
plt.figure(figsize=(8, 6))
plt.plot(r, u_solution, label="u(r)", color="green")
plt.xlabel("r")
plt.ylabel("u(r)")
plt.title("Solution u(r)")
plt.legend()
plt.grid(True)
plt.show()


# Plot z(r) with the boundary condition
plt.figure(figsize=(8, 6))
plt.plot(r, z_solution, label="z(r)", color="blue")
plt.axhline(boundary_condition_z_R, color="green", linestyle="--", label=f"Boundary Condition z(R) = {boundary_condition_z_R:.2f}")
plt.xlabel("r")
plt.ylabel("z(r)")
plt.title("Value Function z(r) with Boundary Condition")
plt.legend()
plt.grid(True)
plt.show()

# Plot optimal production rate with the asymptote
asymptote_value = 1
plt.figure(figsize=(8, 6))
plt.plot(r, optimal_production_rate, label="Optimal Production Rate", color="orange")
plt.axhline(asymptote_value, color="purple", linestyle="--", label=f"Asymptote 1 = {asymptote_value:.2f}")
plt.xlabel("r")
plt.ylabel("Optimal Production Rate")
plt.title("Optimal Production Rate Adjusted for Demand")
plt.legend()
plt.grid(True)
plt.show()

# Plot inventory trajectories
plt.figure(figsize=(10, 6))
if N > 6:
    for i in range(6):  # Plot only the first six products
        plt.plot(time, y_trajectories[:, i], label=f"y_{i + 1}(t)")
    plt.text(0.5, R - 1, "Only first six products are shown", color="red", fontsize=10)
else:
    for i in range(N):
        plt.plot(time, y_trajectories[:, i], label=f"y_{i + 1}(t)")
plt.axhline(R, color='red', linestyle='--', label="Threshold R")
plt.xlabel("Time")
plt.ylabel("Inventory Levels")
plt.title("Inventory Trajectories")
plt.legend()
plt.grid(True)
plt.show()
\end{lstlisting}

\end{document}